\documentclass[a4paper,11pt]{article}
\usepackage[utf8]{inputenc}

\usepackage{amssymb}
\usepackage{latexsym}
\usepackage{amsfonts}
\usepackage{amsmath}
\usepackage{amsthm}
\usepackage{hyperref}
\usepackage{tikz}
\usepackage{tensor}
\usepackage{mathtools}
\usepackage[left=1.2in,right=1.2in,top=1.2in,bottom=1.2in]{geometry}

\usepackage[scale=0.95]{tgheros}
\usepackage[T1]{fontenc}

\DeclareMathOperator{\Rep}{Rep}
\DeclareMathOperator{\Diag}{Diag}

\DeclareMathOperator{\PSL}{PSL}

\DeclareMathOperator{\PSiL}{P\Sigma L}
\DeclareMathOperator{\PGL}{PGL}
\DeclareMathOperator{\PSU}{PSU}

\DeclareMathOperator{\Sp}{Sp}

\DeclareMathOperator{\Ree}{Ree}
\DeclareMathOperator{\Co}{Co}
\DeclareMathOperator{\HS}{HS}

\DeclareMathOperator{\AGaL}{A\Gamma L}
\DeclareMathOperator{\AGL}{AGL}
\DeclareMathOperator{\GL}{GL}

\DeclareMathOperator{\PermAut}{PermAut}
\DeclareMathOperator{\Aut}{Aut}

\DeclareMathOperator{\diff}{diff}

\DeclareMathOperator{\alt}{A}
\DeclareMathOperator{\s}{S}
\DeclareMathOperator{\Sym}{Sym}
\DeclareMathOperator{\soc}{soc}
\DeclareMathOperator{\mg}{M}

\renewcommand{\b}{\underline}

\renewcommand{\b}{\mathbf}

\renewcommand{\leq}{\leqslant}
\renewcommand{\geq}{\geqslant}

\newcommand{\C}{\mathcal C}
\newcommand{\D}{\mathcal{D}}
\newcommand{\F}{\mathbb{F}}
\newcommand{\G}{\mathcal{G}}
\renewcommand{\H}{\mathcal H}

\renewcommand{\P}{\mathcal P}

\renewcommand{\S}{\mathcal{S}}

\newcommand{\Z}{\mathbb{Z}}
\newcommand{\RM}{\mathcal{RM}}
\newcommand{\PRM}{\mathcal{PRM}}

\newcommand{\nsub}{\trianglelefteq}

\newcommand{\Xpi}{X_{{\b 0},i}^{Q_i^\times}}

\DeclareMathOperator{\wt}{wt}
\DeclareMathOperator{\supp}{supp}

\theoremstyle{plain}
\newtheorem{lemma}{Lemma}
\newtheorem{theorem}[lemma]{Theorem}

\newtheorem{proposition}[lemma]{Proposition}
\theoremstyle{definition}
\newtheorem{definition}[lemma]{Definition}

\newtheorem{remark}[lemma]{Remark}

\numberwithin{equation}{section}
\numberwithin{lemma}{section}

\begin{document}

\title{
 Minimal Binary $2$-Neighbour-Transitive Codes\thanks{The first author sincerely thanks Michael Giudici for reading first drafts of this work and is grateful for the support of an Australian Research Training Program Scholarship and a University of Western Australia Safety-Net Top-Up Scholarship. The research forms part of Australian Research Council Project FF0776186.}
}

\author{
 Daniel R. Hawtin$^1$
 \and
 Cheryl E. Praeger$^2$
}

\date{
 \small{
  \emph{
   $^1$School of Science \& Environment (Mathematics),\\
   Memorial University of Newfoundland, Grenfell Campus,\\ 
   Corner Brook, NL, A2H 5G5, Canada.\\
   \href{mailto:dan.hawtin@gmail.com}{dan.hawtin@gmail.com}\\
   \vspace{0.25cm}
  }
 }
 \small{
  \emph{
   $^2$Centre for the Mathematics of Symmetry and Computation,\\
   The University of Western Australia,\\ 
   Crawley, WA, 6009, Australia.\\
   \href{mailto:cheryl.praeger@uwa.edu.au}{cheryl.praeger@uwa.edu.au}\\
   \vspace{0.25cm}
  }
 }
 \today
}

\maketitle
\begin{abstract}
 The main result here is a characterisation of binary \emph{$2$-neighbour-transitive} codes with minimum distance at least $5$ via their minimal subcodes, which are found to be generated by certain designs. The motivation for studying this class of codes comes primarily from their relationship to the class of \emph{completely regular} codes. The results contained here yield many more examples of $2$-neighbour-transitive codes than previous classification results of families of $2$-neighbour-transitive codes. In the process, new lower bounds on the minimum distance of particular sub-families are produced. Several results on the structure of $2$-neighbour-transitive codes with arbitrary alphabet size are also proved. The proofs of the main results apply the classification of minimal and pre-minimal submodules of the permutation modules over $\F_2$ for finite $2$-transitive permutation groups.
\end{abstract}

\section{Introduction}

A subset $C$ of the vertex set $V\varGamma$ of the Hamming graph $\varGamma=H(m,q)$ is a called \emph{code}, the elements of $C$ are called \emph{codewords}, and the subset $C_i$ of $V\varGamma$ consisting of all vertices of $H(m,q)$ having nearest codeword at Hamming distance $i$ is called the \emph{set of $i$-neighbours} of $C$. The classes of \emph{completely regular} and \emph{$s$-regular} codes were introduced by Delsarte \cite{delsarte1973algebraic} as a generalisation of \emph{perfect} codes. The definition of a completely regular code $C$ involves certain combinatorial regularity conditions on the \emph{distance partition} $\{C,C_1,\ldots, C_\rho\}$ of $C$, where $\rho$ is the \emph{covering radius}. The current paper concerns the algebraic analogues, defined directly below, of the classes of completely regular and $s$-regular codes. Note that the group $\Aut(C)$ is the setwise stabiliser of $C$ in the full automorphism group of $H(m,q)$.

\begin{definition}\label{sneighbourtransdef}
 Let $C$ be a code in $H(m,q)$ with covering radius $\rho$, let $s\in\{1,\ldots,\rho\}$, and $X\leq\Aut(C)$. Then $C$ is said to be
 \begin{enumerate}
  \item \emph{$(X,s)$-neighbour-transitive} if $X$ acts transitively on each of the sets $C,C_1,\ldots, C_s$,
  \item \emph{$X$-neighbour-transitive} if $C$ is $(X,1)$-neighbour-transitive, 
  \item \emph{$X$-completely transitive} if $C$ is $(X,\rho)$-neighbour-transitive, and,
  \item \emph{$s$-neighbour-transitive}, \emph{neighbour-transitive}, or \emph{completely transitive}, respectively, if $C$ is $(\Aut(C),s)$-neighbour-transitive, $\Aut(C)$-\emph{neighbour-transitive}, or $\Aut(C)$-\emph{completely transitive}, respectively.
 \end{enumerate}
\end{definition}

A variant of the above concept of complete transitivity was introduced for linear codes by Sol{\'e} \cite{sole1987completely}, with the above definition first appearing in \cite{Giudici1999647}. Note that non-linear completely transitive codes do indeed exist; see \cite{gillespie2012nord}. Completely transitive codes form a subfamily of completely regular codes, and $s$-neighbour transitive codes are a sub-family of $s$-regular codes, for each $s$. It is hoped that studying $2$-neighbour-transitive codes will lead to a better understanding of completely transitive and completely regular codes. Indeed a classification of $2$-neighbour-transitive codes would have as a corollary a classification of completely transitive codes. 

The main result of the present paper, stated below, provides a characterisation of binary $2$-neighbour-transitive codes with minimum distance at least $5$. 

\begin{theorem}\label{binaryx2ntchar}
 Let ${\mathcal C}$ be a binary code in $H(m,2)$ with minimum distance at least $5$. Then $\C$ is $2$-neighbour-transitive if and only if one of the following holds:
 \begin{enumerate}
  \item ${\mathcal C}$ is the binary repetition code with minimum distance $m$;
  \item ${\mathcal C}$ is one of the following codes (see \cite[Definition~4.1]{ef2nt}): 
  \begin{enumerate}
   \item the Hadamard code with $m=12$ and minimum distance $6$,
   \item the punctured Hadamard code with $m=11$ and minimum distance $5$,
   \item the even weight subcode of the punctured Hadamard code with $m=11$ and minimum distance $6$; or,
  \end{enumerate}
  \item there exists a linear subcode $C \subseteq {\mathcal C}$ with minimum distance $\delta$ and a subgroup $X_{\b 0}\leq \Aut(C)_{\b 0}$, where $C$, $X_{\b 0}$, $m$ and $\delta$ are as in Table~\ref{binarytable}, such that
  \begin{enumerate}
   \item $C$ is $(X,2)$-neighbour-transitive, where $X=T_C\rtimes X_{\b 0}\leq \Aut(\C)$, and,
   \item $\C$ is the union of a set $\S$ of cosets of $C$, and $\Aut(\C)$ acts transitively on $\S$.
  \end{enumerate}
 \end{enumerate}
\end{theorem}

\begin{remark}\label{psuremark}
 Note that, in line 9 of Table~\ref{binarytable} it is an open problem as to whether certain values of $r$ correspond to a code having minimum distance $4$. For $\soc(X_{\b 0})=\PSU_3(r)$, if $r\equiv 3\pmod{4}$ then $C$ is self-orthogonal (that is, $C\subseteq C^\perp$), whilst this is not the case when $r\equiv 1\pmod{4}$. Thus, when $\soc(X_{\b 0})=\PSU_3(r)$, there are two examples of minimal $(X,2)$-neighbour-transitive codes if $r\equiv 1\pmod{4}$, one for each of lines 8 and 9, when $\soc(X_{\b 0})=\PSU_3(r)$, but only one, in line 8, when $r\equiv 3\pmod{4}$ (see the proof of Theorem~\ref{binarygroups}). 
\end{remark}

\begin{table}
 \begin{center}
 \begin{tabular}{rccccc}
  line & $\soc(X_{\b 0})$ & $m$ & $\delta$ & $k$ & conditions\\
  \hline
  1 & $\Z_p^d$ & $r=p^d$ & $\geq \sqrt{r-1}+1$ & $\frac{r-1}{2}$ & $23\leq r\equiv 7 \pmod{8}$ \\
   & &  &  &  & $2$-hom. not $2$-trans.\\
  
  2 & $\Z_2^t$ & $2^t$ & $2^{t-1}$ & $t+1$ & $t\geq 4$, $2$-trans. \\
  
  3 & $\PSL_t(2^k)$ & $\frac{2^{kt}-1}{2^k-1}$ & $\geq\frac{2^{k(t-1)}-1}{2^k-1}+1$ & $t^k$ & $t\geq 3$, $(k,t)\neq (1,3)$ \\
  
  4 & $\alt_7$ & $15$ & $8$ & $4$ & - \\
  
  5 & $\PSL_2(r)$ & $r+1$ & $\geq \sqrt{r}+1$ & $\frac{r+1}{2}$ & $23\leq r\equiv\pm 1\pmod 8$ \\
   & &  &  &  & not $3$-trans.\\
  
  6 & $\Sp_{2t}(2)$ & $2^{2t-1}-2^{t-1}$ & $2^{2t-2}-2^{t-1}$ & $2t+1$ & $t\geq 3$ \\
  7 & $\Sp_{2t}(2)$ & $2^{2t-1}+2^{t-1}$ & $2^{2t-2}$ & $2t+1$ & $t\geq 3$ \\
  
  8 & $\PSU_3(r)$ & $r^3+1$ & $\geq r^2+1$ & $r^2-r+1$ & $r$ is odd \\
  9 & $\PSU_3(r)$ & $r^3+1$ & $\geq 4$ & $r^3-r^2+r$ & $r\equiv 1\pmod{4}$ \\
  
  10 & $\Ree(r)$ & $r^3+1$ & $\geq r^{3/2}+1$ & $r^2-r+1$ & $r\geq 3$ \\
  
  11 & $\mg_{22}$ & $22$ & $8$ & $10$ & - \\
  12 & $\mg_{23}$ & $23$ & $8$ & $11$ & - \\
  13 & $\mg_{24}$ & $24$ & $8$ & $12$ & - \\
  
  14 & $\HS$ & $176$ & $\geq 50$ & $21$ & - \\
  
  15 & $\Co_3$ & $276$ & $100$ & $23$ & - \\
  \hline
 \end{tabular}
 \caption{Parameters for the ``minimal'' binary linear-$(X,2)$-neighbour-transitive codes $C$ in $H(m,2)$ of  Part 3 in Theorem~\ref{binaryx2ntchar}, where $k$ is the dimension of $C$, $X=T_C\rtimes X_{\b 0}$, and the minimum distance $\delta$ of $C$ satisfies $5\leq\delta< m$. (See Theorem~\ref{binarygroups}, and also Remark~\ref{psuremark}. Note that $C$ is given in the relevant part of the proof of Theorem~\ref{binarygroups}.)}
 \label{binarytable}
 \end{center}
\end{table}

The parameters of perfect codes over prime power alphabets have been classified, and codes satisfying these parameters found; see \cite{tietavainen1973nonexistence} or \cite{Zinoviev73thenonexistence}. In contrast, the class of completely regular codes is vast, with similar classification results remaining an active area of research. Several recent results have been obtained by Borges et al.~in \cite{borges2000nonexistence,borgesrho1,borges2012new,Borges201468}. For a survey of results on completely regular codes see \cite{borges2017survey}.

Completely-transitive codes have been studied in \cite{Borges201468,Gill2017}, for instance. Neighbour-transitive codes are investigated in  \cite{ntrcodes,gillespieCharNT,gillespiediadntc}. The class of $2$-neighbour-transitive codes is the subject of \cite{ef2nt,aas2nt,ondimblock}, and the present work comprises part of the first author's PhD thesis \cite{myphdthesis}. Codes with $2$-transitive actions on the entries of the Hamming graph have been used to find certain families of codes that achieve capacity on erasure channels \cite{Kudekar:2016:RCA:2897518.2897584}, and many $2$-neighbour-transitive codes indeed admit such actions (see Proposition~\ref{ihom}).

The study of $2$-neighbour-transitive codes has been partitioned into three subclasses, as per the following definition. For definitions and notation see Section~\ref{prelimsect}.

\begin{definition}\label{efaasaadef}
 Let $C$ be a code in $H(m,q)$, $X\leq\Aut(C)$ and $K$ be the kernel of the action of $X$ on the set of entries $M$. Then $C$ is 
 \begin{enumerate}
  \item \emph{$X$-entry-faithful} if $X$ acts faithfully on $M$, that is, $K=1$,
  \item \emph{$X$-alphabet-almost-simple} if $K\neq 1$, $X$ acts transitively on $M$, and $X_i^{Q_i}$ is a $2$-transitive almost-simple group, and,
  \item \emph{$X$-alphabet-affine} if $K\neq 1$, $X$ acts transitively on $M$, and $X_i^{Q_i}$ is a $2$-transitive affine group.
 \end{enumerate}
\end{definition}

Note that Proposition~\ref{ihom} and the fact that every $2$-transitive group is either affine or almost-simple (see \cite[Section 154]{burnside1911theory}) ensure that every $2$-neighbour-transitive code satisfies precisely one of the cases in Definition~\ref{efaasaadef}. 

Those $2$-neighbour transitive codes that are also $X$-entry-faithful, for some automorphism group $X$, and have minimum distance at least $5$ are classified in \cite{ef2nt}; while those that are $X$-alphabet-almost-simple, for some automorphism group $X$, and have minimum distance at least $3$ are classified in \cite{aas2nt}. As a consequence of these results, the proof of Theorem~\ref{binaryx2ntchar} amounts to finding all minimal subcodes of binary $2$-neighbour-transitive and $X$-alphabet-affine codes, for some automorphism group $X$, having minimum distance at least $5$. Note that such codes having a binary repetition subcode are classified in \cite{ondimblock}. 

Table~\ref{binarytable} gives certain parameters of the minimal binary $2$-neighbour-transitive and $X$-alphabet-affine codes with minimum distance $\delta$ satisfying $5\leq\delta\leq m$. While $\delta$ is not explicitly found in every case, bounds are given for those cases where $\delta$ is not known. In particular, in line 10 of Table~\ref{binarytable}, where $\soc(X_{0})=\Ree(r)$, the bound of $\delta\geq r^{3/2}+1$, on the minimum distance $\delta$, is new. This bound follows from Lemma~\ref{mindesigncodedist}, which shows that if the minimum weight codewords of $C$ form a $2$-$(m,\delta,\lambda)$ design and $C$ is self-orthogonal ($C\subseteq C^\perp$), then $\delta\geq \sqrt{m-1}+1$.

Section~\ref{prelimsect} introduces the notation used in the remaining sections. In Section~\ref{extensionsection} it is shown that if $C$ is $X$-alphabet-affine, so that $q=p^d$ for some prime $p$, then $C$ contains an $\F_p X_{\b 0}$-module as a block of imprimitivity for the action of $X$ on $C$. Restricting to the case $q=2$, a strategy of characterisation via minimal submodules is then used in Section~\ref{binarycasesec} for the proof of Theorem~\ref{binaryx2ntchar}. The proof makes use of results on $2$-transitive permutation modules; see, in particular, \cite{Ivanov1993,mortimer1980modular}.


\section{Notation and preliminaries}\label{prelimsect}

Let the \emph{set of entries} $M$ and the \emph{alphabet} $Q$ be sets of sizes $m$ and $q$, respectively, both integers at least $2$. The vertex set $V\varGamma$ of a Hamming graph $\varGamma=H(m,q)$ consists of all functions from the set $M$ to the set $Q$, usually expressed as $m$-tuples. Let $Q_i\cong Q$ be the copy of the alphabet in the entry $i\in M$ so that the vertex set of $H(m,q)$ can be identified with the product 
\[
 V\varGamma=\prod_{i\in M}Q_i.
\] 
An edge exists between two vertices if and only if they differ as $m$-tuples in exactly one entry. Note that $S^\times$ will denote the set $S\setminus \{0\}$ for any set $S$ containing $0$. In particular, $Q$ will usually be a vector-space here, and hence contains the zero vector. A code $C$ is a subset of $V\varGamma$. If $\alpha$ is a vertex of $H(m,q)$ and $i\in M$ then $\alpha_i$ refers to the value of $\alpha$ in the $i$-th entry, that is, $\alpha_i\in Q_i$, so that $\alpha=(\alpha_1,\ldots,\alpha_m)$ when $M=\{1,\ldots,m\}$.  For more in depth background material on coding theory see \cite{cameron1991designs} or \cite{macwilliams1978theory}.

Let $\alpha,\beta$ be vertices and $C$ be a code in a Hamming graph $H(m,q)$ with $0\in Q$ a distinguished element of the alphabet. A summary of important notation regarding codes in Hamming graphs is contained in Table~\ref{hammingnotation}.
\begin{table}
 \begin{center}
 \begin{tabular}{cp{7 cm}}
  Notation & Explanation\\
  \hline
  $\b 0$ & vertex with $0$ in each entry\\
  
  $(a^k,0^{m-k})$ & vertex with $a\in Q$ in the first $k$ entries and $0$ otherwise\\
  
  $\diff(\alpha,\beta)=\{i\in M\mid \alpha_i\neq\beta_i\}$ & set of entries in which $\alpha$ and $\beta$ differ\\
  
  $\supp(\alpha)=\{i\in M\mid \alpha_i\neq 0\}$ & support of $\alpha$\\
  
  $\wt(\alpha)=|\supp(\alpha)|$ & weight of $\alpha$\\
  
  $d(\alpha,\beta)=|\diff(\alpha,\beta)|$ & Hamming distance\\
  
  $\varGamma_s(\alpha)=\{\beta\in V\varGamma \mid d(\alpha,\beta)=s\}$ & set of $s$-neighbours of $\alpha$\\
  
  $\delta=\min\{d(\alpha,\beta)\mid \alpha,\beta\in C,\alpha\neq\beta\}$ & minimum distance of $C$\\
  
  $d(\alpha,C)=\min\{d(\alpha,\beta) \mid \beta\in C\}$ & distance from $\alpha$ to $C$\\
  
  $\rho =\max\{d(\alpha,C)\mid\alpha\in V\varGamma\}$ & covering radius of $C$\\
  
  $C_s=\{\alpha\in V\varGamma \mid d(\alpha,C)=s\}$ & set of $s$-neighbours of $C$\\
  
  $\{C=C_0,C_1,\ldots, C_\rho\}$ & distance partition of $C$\\
    
  \hline
 \end{tabular}
 \caption{Hamming graph notation.}
 \label{hammingnotation}
 \end{center}
\end{table}

Note that if the minimum distance of a code $C$ satisfies $\delta\geq 2s$, then the set of $s$-neighbours $C_s$ satisfies $C_s=\cup_{\alpha\in C}\varGamma_s(\alpha)$ and if $\delta\geq 2s+1$ this is a disjoint union. This fact is crucial in many of the proofs below; it is often assumed that $\delta\geq 5$, in which case every element of $C_2$ is distance $2$ from a unique codeword.

A \emph{linear} code is a code $C$ in $H(m,q)$ with alphabet $Q=\F_q$ a finite field, so that the vertices of $H(m,q)$ from a vector space $V$, such that $C$ is an $\F_q$-subspace of $V$. Given $\alpha,\beta\in V$, the usual inner product is given by $\langle\alpha,\beta\rangle=\sum_{i\in M}\alpha_i\beta_i$. The dual of a linear code is defined below.

\begin{definition}\label{dualcodedef}
 Let $C$ be a linear code in $H(m,q)$ with vertex set $V$. Then the \emph{dual} code of $C$ is $C^\perp=\{\beta\in V\mid \forall \alpha\in C,\langle\alpha,\beta\rangle=0\}$.
\end{definition}


\subsection{Automorphisms of a Hamming graph}\label{hamminggraphautoprelim}

The automorphism group $\Aut(\varGamma)$ of the Hamming graph is the semi-direct product $B\rtimes L$, where $B\cong \Sym(Q)^m$ and $L\cong \Sym(M)$ (see \cite[Theorem 9.2.1]{brouwer}). Note that $B$ and $L$ are called the \emph{base group} and the \emph{top group}, respectively, of $\Aut(\varGamma)$. Since we identify $Q_i$ with $Q$, we also identify $\Sym(Q_i)$ with $\Sym(Q)$. If $h\in B$ and $i\in M$ then $h_i\in \Sym(Q_i)$ is the image of the action of $h$ in the entry $i\in M$. Let $h\in B$, $\sigma\in L$ and $\alpha\in V\varGamma$. Then $h$ and $\sigma$ act on $\alpha$ explicitly via: 
\begin{equation*}
\alpha^h =(\alpha_1^{h_1},\ldots,\alpha_m^{h_m})\quad\text{and}\quad
\alpha^\sigma=(\alpha_{1{\sigma^{-1}}},\ldots,\alpha_{m{\sigma^{-1}}}).
\end{equation*}
The automorphism group of a code $C$ in $\varGamma=H(m,q)$ is $\Aut(C)=\Aut(\varGamma)_C$, the setwise stabiliser of $C$ in $\Aut(\varGamma)$. 

A group with an element or set appearing as a subscript denotes a setwise stabiliser subgroup, and if the subscript is a set in parantheses it is a point-wise stabiliser subgroup. A group with a set appearing as a superscript denotes the subgroup induced in the symmetric group on the set by the group. (For more background and notation on permutation groups see, for instance, \cite{dixon1996permutation}.) In particular, let $X$ be a subgroup of $\Aut(\varGamma)$. Then the \emph{action of $X$ on entries} is the subgroup $X^M$ of $\Sym(M)$ induced by the action of $X$ on $M$. Note that the pre-image of an element of $X^M$ does not necessarily fix any vertex of $H(m,q)$. The kernel of the action of $X$ on entries is denoted $K$ and is precisely the subgroup of $X$ fixing $M$ point-wise, that is, $K=X_{(M)}=X\cap B$. The subgroup of $\Sym(Q_i)$ induced on the alphabet $Q_i$ by the action of the stabiliser $X_i\leq X$ of the entry $i\in M$ is denoted $X_i^{Q_i}$. When $X^M$ is transitive on $M$, the group $X_i^{Q_i}$ is sometimes referred to as the \emph{action on the alphabet}. 

Given a group $H\leq \Sym(Q)$ an important subgroup of $\Aut(\varGamma)$ is the \emph{diagonal} group of $H$, denoted $\Diag_m(H)$, where an element of $H$ acts the same in each entry. Formally, $\Diag_m(H)=\{g\in B\mid \exists h\in H \text{ such that, }\forall i\in M, g_i=h \}$.

It is worth mentioning that coding theorists often consider more restricted groups of automorphisms, such as the group $\PermAut(C)=\{\sigma\mid h\sigma\in\Aut(C), h=1\in B, \sigma\in L\}$. The elements of this group are called \emph{pure permutations} on the entries of the code. 

Two codes $C$ and $C'$ in $H(m,q)$ are said to be \textit{equivalent} if there exists some $x\in \Aut(\varGamma)$ such that $C^x=\{\alpha^x\mid\alpha\in C\}=C'$. Equivalence preserves many of the important properties in coding theory, such as minimum distance and covering radius, since $\Aut(\varGamma)$ preserves distances in $H(m,q)$.

A \emph{linear representation} or \emph{projective representation} of a group $G$ is a homomorphism $\phi$ from $G$ into $\GL(V)$ or $\PGL(V)$, respectively, for some vector space $V$. If $F$ is the field underlying $V$, then the vector space $V$ is called an \emph{$FG$-module} and any subspace $U\leq V$ such that $U^G=U$, where the action of $G$ is induced by $\phi$, is called an \emph{$FG$-submodule} of $V$, and is itself an $FG$-module. Either of $F$ or $G$ may be omitted if the context is clear. An $FG$-module $V$ is called \emph{irreducible} if it contains no non-trivial submodule, that is, if $U\leq V$ such that $U^G=U$ then $U=V$ or $U=0$. Note that these definitions are usually given in a slightly more general form, but suffice for the purposes of later chapters. For further background see \cite{james2001representations}, for instance.

Suppose that $C$ is an $\F_q$-linear code, so that $V\varGamma\cong\F_q^m$, and $C$ is $(X,s)$-neighbour-transitive with $X=T_C\rtimes X_{\b 0}$. Then $X_{\b 0}$ is naturally embedded in $\GL(V\varGamma)$ and thus $V\varGamma$ is an $\F_q X_{\b 0}$-module and $C$ is a $\F_q X_{\b 0}$-submodule of $V\varGamma$.

\subsection{\texorpdfstring{$s$}{s}-Neighbour-transitive codes}\label{sntr}

This section presents preliminary results regarding $(X,s)$-neighbour-transitive codes, defined in Definition~\ref{sneighbourtransdef}. The next results give certain $2$-homogeneous and $2$-transitive actions associated with an $(X,2)$-neighbour-transitive code.

\begin{proposition}\cite[Proposition~2.5]{ef2nt}\label{ihom} 
 Let $C$ be an $(X,s)$-neighbour-transitive code in $H(m,q)$ with minimum distance $\delta$. Then for $\alpha\in C$ and $i\leq\min\{s,\lfloor\frac{\delta-1}{2}\rfloor\}$, the stabiliser $X_\alpha$ fixes setwise and acts transitively on $\varGamma_i(\alpha)$.  In particular, $X_\alpha$ acts $i$-homogeneously on $M$.    
\end{proposition}

%

\begin{proposition}\cite[Proposition~2.7]{ef2nt}\label{x12trans} 
 Let $C$ be an $(X,1)$-neighbour-transitive code in $H(m,q)$ with minimum distance $\delta\geq 3$ and $|C|>1$.  Then $X_i^{Q_i}$ acts $2$-transitively on $Q_i$ for all $i\in M$.
\end{proposition}

The next result gives information about the order of the stabiliser of a codeword in the automorphism group of a $2$-neighbour-transitive code.

\begin{lemma}\cite[Lemma~2.10]{ef2nt}\label{codeorder}
 If $C$ is an $(X,2)$-neighbour transitive code in $H(m,q)$ with $\delta\geq 5$ and $\b 0\in C$, then $\binom{m}{2} (q-1)^2$ divides $|X_{\b 0}|$, and hence $|X|$. In particular, if $|X_{\b 0}|=m(m-1)/2$ then $q=2$.
\end{lemma}

The concept of an $(X,2)$-neighbour-transitive extension was introduced in \cite{ondimblock}.

\begin{definition}\label{xntextensiondef}
 Let $q=p^d$, $V=\F_p^{dm}$ and $W$ be a non-trivial $\F_p$-subspace of $V$. Identify $V$ with the vertex set of the Hamming graph $H(m,q)$. An \emph{$(X,2)$-neighbour-transitive extension} of $W$ is an $(X,2)$-neighbour-transitive code $C$ containing $\b 0$ such that $T_W\leq X$ and $K=K_W$, where $K=X\cap B$, $T_W$ is the group of translations by elements of $W$ and $K_W$ is the stabiliser of $W$ in $K$. Note that $T_W\leq X$ and $\b 0\in C$ means that $W\subseteq C$. If $C\neq W$ then the extension is said to be \emph{non-trivial}.
\end{definition}

The main result of \cite{ondimblock} is stated below. The importance of this result, in the present context, is that it classifies all binary $2$-neighbour-transitive codes with minimum distance at least $5$ that have a repetition subcode. 

\begin{theorem}\cite[Theorem~1.1]{ondimblock}\label{onedimensionaltheorem}
 Let $V=\F_p^{dm}$ be the vertex set of the Hamming graph $H(m,p^d)$ and $C$ be an $(X,2)$-neighbour-transitive extension of $W$ with $\delta\geq 5$, where $W$ is an $\F_p$-subspace of $V$ with $\F_p$-dimension $k\leq d$. Then $p=2$, $d=1$ and one of the following holds:
 \begin{enumerate}
  \item $C=W$ is the binary repetition code, with $\delta=m$,
  \item $C=\H$, where $\H$ is the Hadamard code of length $12$, as in \cite[Definition~4.1]{ef2nt}, with $\delta=6$, or,
  \item $C=\P$, where $\P$ is the punctured code of the Hadamard code of length $12$, as in \cite[Definition~4.1]{ef2nt}, with $\delta=5$.
 \end{enumerate}
 Moreover, in each case $C$ is an $(\Aut(C),2)$-neighbour-transitive extension of the binary repetition code in the appropriate Hamming graph.
\end{theorem}

The concept of a design, introduced below, comes up frequently in coding theory. Let $\alpha\in H(m,q)$ and $0\in Q$. A vertex $\nu$ of $H(m,q)$ is said to be \emph{covered} by $\alpha$ if $\nu_i=\alpha_i$ for every $i\in M$ such that $\nu_i\neq 0$. A binary design, obtained by setting $q=2$ in the below definition, is usually defined as a collection of subsets of some ground set, satisfying equivalent conditions. We refer to the latter structures as combinatorial designs. In particular, the concept of covering a vertex just described, corresponds to containment of a subset.

\begin{definition}\label{desdef}
 A \emph{$q$-ary $s$-$(v,k,\lambda)$ design} in $\varGamma=H(m,q)$ is a subset $\mathcal{D}$ of vertices of $\varGamma_k(\b 0)$ (where $k\geq s$) such that each vertex $\nu \in\varGamma_s(\b 0)$ is covered by exactly $\lambda$ vertices of $\mathcal{D}$. When $q=2$, $\D$ is simply the set of characteristic vectors of a combinatorial $s$-design. The elements of $\D$ are called \emph{blocks}.
\end{definition}

The following equations can be found, for instance, in \cite{stinson2004combinatorial}. Let $\D$ be a binary $s$-$(v,k,\lambda)$ design with $|\D|=b$ blocks and let $r$ be the number of blocks incident with a point. Then $vr=bk$, $r(k-1)=\lambda(v-1)$ and \[b=\frac{v(v-1)\cdots(v-s+1)}{k(k-1)\cdots(k-s+1)}\lambda.\]

\begin{lemma}\cite[Lemma~2.16]{ef2nt}\label{design}
 Let $C$ be an $(X,s)$-neighbour transitive code in $H(m,q)$. Then $C$ is $s$-regular. Moreover, if $ \b 0\in C$ and $\delta\geq 2s$ then for each $k\leq m$, the set of codewords of weight $k$ forms a $q$-ary $s$-$(m,k,\lambda)$ design, for some $\lambda$.
\end{lemma}

\section{Modules as blocks of imprimitivity}\label{extensionsection}

This section investigates the structure of an $X$-alphabet-affine code $C$ via the group $X$. When the kernel $K$ of the action of $X$ on $M$ does not act transitively on $C$ a system of imprimitivity can be identified. For a group $G$ and a prime $p$ the \emph{$p$-core} $O_p(G)$ is the largest normal $p$-subgroup of $G$.

\begin{lemma}\label{opkiist}
 Let $C$ be an $X$-alphabet-affine code in $H(m,q)$, where $q=p^d$ for some prime $p$, and $i\in M$. Then the $p$-core $O_p\left(K^{Q_i}\right)$ of $K^{Q_i}$ is the unique minimal normal subgroup of $X_i^{Q_i}$.
\end{lemma}

\begin{proof}
 The fact that $C$ is $X$-alphabet-affine implies, by Definition~\ref{efaasaadef}, that $X^M$ is transitive, $K\neq 1$ and $X_i^{Q_i}$ is a $2$-transitive affine group. It follows that $X_i$ and $X_j$ are conjugate in $X$ for all $i,j\in M$. Since $K\vartriangleleft X$ it follows that $K^{Q_i}\vartriangleleft X_i^{Q_i}$ for each $i\in M$, and since $X^M$ is transitive it follows that $K^{Q_i}$ is isomorphic to $K^{Q_j}$. Now $K\leq \prod_{i\in M}K^{Q_i}$ which implies, since $K\neq 1$, that $K^{Q_i}\neq 1$ for all $i\in M$. By hypothesis, we have $X_i^{Q_i}\cong T_i\rtimes G_0$, where $T_i\cong \Z_p^d$ is the minimal normal subgroup of $X_i^{Q_i}$ and $G_0$ acts transitively on $Q_i^\times$. Hence $K^{Q_i}$ contains $T_i$ as a normal subgroup. Since $T_i$ is a normal $p$-subgroup of $K^{Q_i}$, it is contained in $O_p\left(K^{Q_i}\right)$. Since $O_p\left(K^{Q_i}\right)$ is a characteristic subgroup of $K^{Q_i}$ it follows that $O_p\left(K^{Q_i}\right)\leq O_p\left(X_i^{Q_i}\right)$. We claim that $O_p\left(X_i^{Q_i}\right)=T_i$, from which the result follows. 
 
 Now $P:=O_p\left(X_i^{Q_i}\right)\cap G_0$ is a normal $p$-subgroup of $G_0$ such that $O_p\left(X_i^{Q_i}\right)=T_i\rtimes P$. Since $G_0$ is transitive on $Q_i^\times$, and since $P\nsub G_0$, all the $P$-orbits have the same length, say $u$, and $u$ divides $|Q_i^\times|=p^d-1$. However, $u$ also divides $|P|$, and so $u$ is a $p$-power. This implies that $u=1$ and hence that $P=1$. Thus $O_p\left(X_i^{Q_i}\right)=T_i$.
\end{proof}

\begin{lemma}\label{opGiszptodm}
 Let $C$ be an $X$-alphabet-affine code in $H(m,q)$ where $q=p^d$. Let $G=\prod_{i\in M}K^{Q_i}$ and $T_i$ be the minimal normal subgroup of $X_i^{Q_i}$. Then $O_p(G)= \prod_{i\in M} T_i$ acts regularly on the vertex set of $H(m,q)$. 
\end{lemma}

\begin{proof}
 Let 
 \[
  T=\prod_{i\in M} T_i\cong \Z_p^{dm}.
 \]
 Then $T$ is a normal $p$-subgroup of $G$ and hence is contained in $N=O_p(G)$. Let $N_i=N^{Q_i}\nsub G^{Q_i}=K^{Q_i}$, for each $i\in M$. Since $N$ is a $p$-group, so is $N_i$. Thus, by Lemma~\ref{opkiist}, $N_i\leq T_i$, and so $N\leq \prod_{i\in M} N_i\leq \prod_{i\in M} T_i=T$ which implies that $T=N$. Finally, $T_i$ acts regularly on $Q_i$, so that $O_p(G)=\prod_{i\in M} T_i$ acts regularly on $\prod_{i\in M} Q_i$, the vertex set of $H(m,q)$.
\end{proof}

Lemma~\ref{orbitofopisamodule} presents the main idea of this section.

\begin{lemma}\label{orbitofopisamodule}
 Let $C$ be an $X$-alphabet-affine code in $H(m,q)$ with $q=p^d$, and $G$, $T_i$ be as in Lemma~\ref{opGiszptodm}. Then $O_p(K)=K\cap O_p(G)$ is normal in $X$, the orbit of the vertex $\b 0$ under $O_p(K)$ is an $\F_p X_{\b 0}$-module and $K=O_p(K)\rtimes K_{\b 0}$.
\end{lemma}

\begin{proof}
 Now $O_p(K)\leq \prod_{i\in M}O_p(K)^{Q_i}$ and for each $i$, $O_p(K)^{Q_i}$ is a normal $p$-subgroup of $X_i^{Q_i}$. Thus $O_p(K)^{Q_i}\leq O_p(K^{Q_i})$ which, by Lemma~\ref{opkiist}, is $T_i$. That is,
 \[
  O_p(K)\leq \prod_{i\in M}O_p(K)^{Q_i}\leq \prod_{i\in M}T_i = O_p(G),
 \]
 by Lemma~\ref{opGiszptodm}. Since $O_p(K)\leq K$ we have $O_p(K)\leq K\cap O_p(G)$, and since $K\cap O_p(G)$ is a $p$-group that is normal in $K$ (as $O_p(G)\nsub G$), equality holds. Now $O_p(K)$ is a characteristic subgroup of $K$ and hence is normal in $X$. In particular, $O_p(K)$ is $X_{\b 0}$-invariant and hence the orbit ${\b 0}^{O_p(K)}$ is an $\F_p X_{\b 0}$-submodule of $T=\prod_{i\in M}T_i\cong \F_p^{dm}$. Now, $G_{\b 0}=\prod_{i\in M}K_{\b 0}^{Q_i^\times}$, so that $O_p(G)\cap G_{\b 0}=1$. Also, $O_p(G)=T$ is transitive on the set of vertices of $H(m,q)$, by Lemma~\ref{opGiszptodm}, so that $G=O_p(G)\rtimes G_{\b 0}$. Thus $K=O_p(K)\rtimes K_{\b 0}$. 
\end{proof}

The next lemma is needed in order to prove Proposition~\ref{orbitof0is2ntmodule}.

\begin{lemma}\label{xijtransonqiqj}
 Let $C$ be an $X$-alphabet-affine and $(X,2)$-neighbour-transitive code in $H(m,q)$ with $\delta\geq 5$, and $C\neq\Rep(m,2)$ if $q=2$, such that the group $T_C$ of translations by codewords of $C$ is contained in $X$. Then $X_{i,j}$ acts transitively on $Q_i\times Q_j$.
\end{lemma}

\begin{proof}
 By Lemma~\ref{design}, the weight $\delta$ codewords form a $q$-ary $2$-$(m,\delta,\lambda)$ design, for some $\lambda>0$. Thus, for all $a\in Q_i^\times$ and $b\in Q_j^\times$, there exists at least one weight $\delta$ codeword $\alpha$ such that $\alpha_i=a$ and $\alpha_j=b$. Hence, for all $a\in Q_i^\times$ and $b\in Q_j^\times$, there exists $h_{a,b}\in T_C$ such that $0^{(h_{a,b})_i}=a$ and $0^{(h_{a,b})_j}=b$. It follows that $T_C$ acts transitively on both of the sets $Q_i$ and $Q_j$, and, since $T_C\leq X_{i,j}$, so does $X_{i,j}$. To complete the proof we will show that the subgroup $S$ of $T_C$ fixing $0\in Q_i$ is transitive on $Q_j$. 
 
 Since $C\neq \Rep(m,2)$ when $q=2$, it follows from \cite[Lemma~2.15]{ef2nt} that $\delta< m$. Thus, there exist $\beta\in C$ and distinct $i',j'\in M$ such that $\beta_{i'}=0$ and $\beta_{j'}=c'$, for some $c'\in Q_{j'}^\times$. By Proposition~\ref{ihom}, $X_{\b 0}$ acts $2$-homogeneously on $M$, so there exists some $\gamma\in C$ such that either $\gamma_i=0$ and $\gamma_j=c$, for some $c\in Q_j^\times$, or $\gamma_i=c$ and $\gamma_j=0$, for some $c\in Q_i^\times$. If $\gamma$ has the second property, then $\gamma^{(h_{-c,c})}$ has the first property. Thus we may, without loss of generality, assume that $\gamma_i=0$ and $\gamma_j=c\neq 0$.  Finally, for any given $b\in Q_j^\times$ define $\gamma'=\gamma^{h_{c,-c}h_{-c,b}}$ and observe that $(\gamma')_i=0+c-c=0$ and $(\gamma')_j=c-c+b=b$. Then the translation $t_{\gamma'}$ lies in $S$ and maps $0$ to $b$ in $Q_j$. Thus $S$ is transitive on $Q_j$.
\end{proof}


Recall the definition of an $(X,2)$-neighbour-transitive extension from Definition~\ref{xntextensiondef}.

\begin{proposition}\label{orbitof0is2ntmodule}
 Let $C$ be an $X$-alphabet-affine and $(X,2)$-neighbour-transitive code in the Hamming graph $H(m,q)$ with $q=p^d$, $\delta\geq 5$ and $\b 0\in C$. Then $C$ is an $(X,2)$-neighbour-transitive extension of $W$, where $W$ is the code formed by the orbit of $\b 0$ under $O_p(K)$, with $K=X\cap B$. It follows that: 
 \begin{enumerate}
  \item $X_W=O_p(K)\rtimes X_{\b 0}$,
  \item $W$ is $X_W$-alphabet-affine, 
  \item $W$ is $(X_W,2)$-neighbour-transitive with minimum distance $\delta_W\geq 5$,
  \item $W$ is an $\F_p X_{\b 0}$-module, and,
  \item if $W\neq \Rep(m,2)$ then $q^2$ divides $|W|$. 
 \end{enumerate}
\end{proposition}

\begin{proof} 
 By Lemma~\ref{orbitofopisamodule}, $W={\b 0}^{O_p(K)}$ is an $\F_p X_{\b 0}$-module, $O_p(K)\nsub X$, and $K=O_p(K)\rtimes K_{\b 0}$. In particular, part 4 is proved. By Lemma~\ref{opGiszptodm}, $O_p(G)\cong\Z_p^{dm}$ acts regularly on the vertex set of $H(m,q)$, from which it follows that $T_W=O_p(K)\nsub X$. Now, $T_W\nsub X$ implies, by \cite[Lemma~3.1]{ondimblock}, that $W$ is a block of imprimitivity for the action of $X$ on $C$. By assumption $W$ contains $\b 0$, which implies that $X_{\b 0}$, and thus $K_{\b 0}$, fixes $W$. Thus $K=K_W$, and so $C$ is an $(X,2)$-neighbour-transitive extension of $W$. As $T_W$ is transitive on $W$, we have $X_W=T_W\rtimes X_{\b 0}$, proving part 1. By \cite[Corollary~3.2]{ondimblock}, $W$ is $(X_W,2)$-neighbour-transitive and $\delta_W\geq 5$, which gives part 3. Now, $T_W$, and hence $K$, is non-trivial, $X_{W,i}^{Q_i}\leq X_{i}^{Q_i}\leq\AGL_d(p)$ and, by Proposition~\ref{ihom}, $X_{\b 0}$ acts transitively on $M$. Since, by part 1, $X_{\b 0}=X_W\cap X_{\b 0}$, and since $X_{{\b 0},i}$ is transitive on $Q_i^\times$, it follows that $X_{W,i}^{Q_i}$ is affine $2$-transitive, and so, by Definition~\ref{efaasaadef}, $W$ is $X_W$-alphabet-affine, establishing part 2.
 
 Suppose $W\neq\Rep(m,2)$ and fix $i,j\in M$ with $i\neq j$. Suppose there are $k$ codewords $\alpha\in W$ such that $\alpha_i=\alpha_j=0$. Lemma~\ref{xijtransonqiqj} implies that $X_{W,i,j}$ acts transitively on $Q_i\times Q_j$. Thus, for any $a\in Q_i$ and $b\in Q_j$ there are the same number $k$ of codewords $\beta\in W$ such that $\beta_i=a$ and $\beta_j=b$. Therefore, $|W|=kq^2$, proving part 5.
\end{proof}

Proposition~\ref{orbitof0is2ntmodule} states that every code that is both $X$-alphabet-affine and $(X,2)$-neighbour-transitive with $\delta\geq 5$ is an $(X,2)$-neighbour-transitive extension of some $\F_p$-subspace $W$ of the vertex set $V\cong\F_p^{dm}$ of $H(m,q)$, where $q=p^d$. Since $W$ must also be $(X_W,2)$-neighbour-transitive, classifying $(X,2)$-neighbour-transitive codes requires knowledge of all subspaces of $V$ which form $2$-neighbour-transitive codes, and all $(X,2)$-neighbour-transitive extensions of them. In fact, given a subgroup $X_{\b 0}$ of $\Aut(\varGamma)_{\b 0}$ such that $X_{\b 0}$ acts transitively on $\varGamma_2(\b 0)$, any $\F_p X_{\b 0}$-submodule $W$, with minimum distance $\delta_W\geq 5$, of $V$ will be $(X_W,2)$-neighbour-transitive, with $X_W=T_W\rtimes X_{\b 0}$. This suggests finding all such groups $X_{\b 0}$ and considering $\F_p X_{\b 0}$-submodules of $V$. The next section investigates this problem for the case $q=2$. 

%
%

\section{Binary linear codes}\label{binarycasesec}

When $q=2$ a code is called \emph{binary} and, for any $X\leq\Aut(\varGamma)$ the groups $\Xpi$ and $K_{\b 0}$ are trivial. This section considers binary codes satisfying the following definition. Since $K_{\b 0}=1$, it follows from the results of the previous section that these codes are the building blocks of binary $(X,2)$-neighbour-transitive codes. Note that the condition $\Diag_m(\F_q^\times)\leq K_{\b 0}$ is trivially satisfied when $q=2$.

\begin{definition}
 A code $C$ in $H(m,q)$ is \emph{linear-$(X,2)$-neighbour-transitive} if $\Diag_m(\F_q^\times)\leq K_{\b 0}$, $T_C\leq X$ and $C$ is $(X,2)$-neighbour-transitive.
\end{definition}

The next lemma shows that a binary linear-$(X,2)$-neighbour-transitive code with $\delta\geq 5$ is a submodule of the $\F_2$-permutation module of a $2$-homogeneous group in its $2$-homogeneous action.

\begin{lemma}\label{binarycodepermsubmodules}
 Let $C$ be an $(X,2)$-neighbour-transitive code with $T_C\leq X$ and minimum distance $\delta\geq 5$ in the Hamming graph $H(m,2)$ with vertex set $V\cong\F_2^m$. Then $C$ is a submodule of $V$, regarded as the permutation module for the $2$-homogeneous action of $X^M\cong X_{\b 0}$ on $M$.
\end{lemma}

\begin{proof}
 If $x=h\sigma\in X_{\b 0}$, with $h\in B$ and $\sigma\in L$, then $h=1$, since $h_i$ fixes $0$ for each $i$, so that $x=\sigma\in L$. Thus $X_{\b 0}\cong X_{\b 0}^M\cong X^M$ acts as pure permutations on entries, so that $V$ may be regarded as the permutation module for the action of $X_{\b 0}$ on $M$. Since $\delta\geq 5$, Proposition~\ref{ihom} implies that this action is $2$-homogeneous. Since $X_{\b 0}$ acts faithfully on $M$ we have that $K_{\b 0}=1$. Thus $K=T_C$ (since $T_C\leq K$ and $K_{\b 0}=1$) and hence $X=T_C\rtimes X_{\b 0}$. Thus, $C$ is a submodule of the $2$-homogeneous $\F_2 X_{\b 0}$-permutation module $V$. 
\end{proof}

The next result, Lemma~\ref{moduleiscode}, is a kind of converse to Lemma~\ref{binarycodepermsubmodules}. Note that if $C$ has minimum distance $\delta=3$ and $C$ is perfect, then $C$ has covering radius $\rho=1$, and is thus \emph{not} $2$-neighbour-transitive, since $C_2$ is empty. 

\begin{lemma}\label{moduleiscode}
 Let $G$ act $2$-homogeneously on a set $M$ of size $m\geq 5$, let $V\cong\F_2^m$ be the permutation module for the action of $G$ on $M$, let $Y$ be the submodule of $V$ consisting of the set of all constant functions, and let $C$ be an $\F_2 G$-submodule of $V$. Then $C$ is a code in $H(m,2)$ with minimum distance $\delta$, $X=T_C\rtimes G\leq \Aut(C)$, and precisely one of the following statements holds:
 \begin{enumerate}
  \item $C=\{\b 0\}$ with $\delta$ undefined.
  \item $\delta=1$ and $C=V$.
  \item $\delta=m$ and $C=Y$ is linear-$(X,2)$-neighbour-transitive.
  \item $\delta=2$ and $C=Y^\perp$, the dual of $Y$ under the standard inner product, and $C$ is $X$-neighbour-transitive.
  \item $\delta=3$, $C$ is a perfect code in $H(m,2)$, and $C$ is $X$-neighbour-transitive.
  \item $4\leq \delta< m$ and $C$ is linear-$(X,2)$-neighbour-transitive.
 \end{enumerate}
\end{lemma}

\begin{proof}
 First, the permutation module $V$ may be regarded as the vertex set of $H(m,2)$ in the natural way, so that $C$ is a code in $H(m,2)$. Also $X=T_C\rtimes G\leq \Aut(C)$, so that $X_{\b 0}=G$. Part 1 holds if and only if $|C|=1$. Assume $|C|\geq 2$ and let $\delta$ be the minimum distance of $C$. Since $T_C$ acts transitively on $C$, and ${\b 0}\in C$, there exists a weight $\delta$ codeword in $C$. Note that the weight $1$ vertices are the characteristic vectors of the subsets of $M$ of size $1$, and the weight $2$ vertices are the characteristic vectors of the subsets of $M$ of size $2$. Since $G$ acts $2$-homogeneously on $M$, it follows that $X_{\b 0}$ acts transitively on each of the sets $\varGamma_s({\b 0})$, for $s=1,2$. Thus, $\varGamma_1({\b 0})$ is a subset of either $C$ or $C_1$, since any weight $1$ vertex is distance $1$ from ${\b 0}\in C$, and $\varGamma_2({\b 0})$ is a subset of either $C$, $C_1$ or $C_2$, since any weight $2$ vertex is distance $2$ from ${\b 0}\in C$. If $\varGamma_2({\b 0})\subseteq C_2$ it immediately follows that $C$ is $(X,2)$-neighbour-transitive, as $T_C$ acts transitively on $C$, and $X_{\b 0}$ acts transitively on $\varGamma_1({\b 0})$ and $\varGamma_2({\b 0})$.
 
 Suppose $\delta=1$. Then there exists some $\alpha,\beta\in C$ such that $d(\alpha,\beta)=1$. Since $T_C$ acts transitively on $C$, it can be assumed that $\beta={\b 0}$. It then follows that $\alpha$ is in $\varGamma_1({\b 0})\cap C$, so that $\varGamma_1({\b 0})\subseteq C$. Thus, every weight $1$ vertex $\nu$ is in $C$, and the translation $t_\nu$ by $\nu$ lies in $X$. Hence $C=V$, as in part 2.
 
 Suppose $\delta=m$. If $\alpha\in V$ with $d({\b 0},\alpha)=m$ it follows that $\alpha_i=1$ for all $i\in M$. As ${\b 0}\in C$, we have part 3, that is $C=Y$. Since $m\geq 5$, we deduce $\varGamma_2({\b 0})\subseteq C_2$, so that $C$ is $(X,2)$-neighbour-transitive, by the argument in the first paragraph of the proof.
 
 Suppose $\delta=2$. Then $\varGamma_1({\b 0})\subseteq C_1$. However, there exists a vertex $\alpha\in \varGamma_2({\b 0})\cap C$, so that $\varGamma_2({\b 0})\subseteq C$ and every weight $2$ codeword is in $C$. Thus $C=Y^\perp$, and part 4 holds.
 
 Suppose $\delta=3$. Again, $\varGamma_1({\b 0})\subseteq C_1$. Now, there exists a weight $3$ vertex $\alpha\in C$ and distinct $i,j,k\in M$ such that $\supp(\alpha)=\{i,j,k\}$. Let $\nu\in\varGamma_2({\b 0})$ such that $\nu_i=\nu_j=1$. Then $d(\alpha,\nu)=1$, so that $\nu\in C_1$. Hence, $\varGamma_2({\b 0})\subseteq C_1$ and so $\varGamma_2({\b 0})\cap C_2=\emptyset$. As $T_C$ acts transitively on $C$, we have $\varGamma_2(\beta)\cap C_2=\emptyset$ for any $\beta\in C$. Hence $C_2=\emptyset$ and $C$ is perfect, giving part 5.
 
 Finally, suppose $4\leq \delta< m$. Since $\delta\geq 4$, for $i=1,2$, every weight $i$ vector is at least distance $i$ from any non-zero codeword. Thus, $\varGamma_1({\b 0})\subseteq C_1$ and $\varGamma_2({\b 0})\subseteq C_2$ so that $C$ is a linear-$(X,2)$-neighbour-transitive code, completing the proof.
\end{proof}

A lower bound for the minimum distance of the dual of a linear code generated by the blocks of certain designs is given in \cite[Lemma~2.4.2]{assmus1994designs}. This result is applied below, and will be used later to provide information about the minimum distance of some binary linear-$(X,2)$-neighbour-transitive codes of interest. 

\begin{lemma}\label{mindesigncodedist}
 Let $C$ be binary linear code in $H(m,2)$ with minimum distance $\delta$ satisfying $3\leq\delta<m$ such that the weight $\delta$ codewords of $C$ form a $2$-$(m,\delta,\lambda)$ design $\D$, and let $C^\perp$ be the dual code of $C$ with minimum distance $\delta^\perp$. (In particular, this is satisfied if there exists $X\leq\Aut(C)$ such that $X_{\b 0}$ acts $2$-homogeneously on $M$.) Then $m-1\leq (\delta-1)(\delta^\perp-1)$. Furthermore, if $C$ is self-orthogonal, that is $C\subseteq C^\perp$, then $\delta\geq \sqrt{m-1}+1$.
\end{lemma}

\begin{proof}
 First, as $3\leq\delta<m$, neither $C$ nor $C^\perp$ is either the repetition code, or its dual. If $X_{\b 0}$ acts $2$-homogeneously on $M$, then each weight $2$ vector, being the characteristic vector of a $2$-subset of $M$, is covered by the same number $\lambda$ of weight $\delta$ elements of $C$. Thus, the set $\varGamma_\delta({\b 0})\cap C$ of all weight $\delta$ codewords of $C$ forms a $2$-$(m,\delta,\lambda)$ design $\D$ for some integer $\lambda$, justifying the statement in parentheses. 
 
 Let $C_2(\D)$ be the code spanned by the blocks of $\D$, considered as characteristic vectors. Now, $C_2(\D)$ is fixed setwise by $X_{\b 0}$ and $C_2(\D)\neq\Rep(m,2)$, since $\D$ contains at least $1$ weight $\delta$ vertex, and $\delta<m$. As $C_2(\D)$ is contained in $C$, it follows that the minimum distance of $C_2(\D)$ is also $\delta$, and $C^\perp$ is contained in $(C_2(\D))^\perp$. Thus, $\delta^\perp$ is bounded below by the minimum distance of $(C_2(\D))^\perp$. Hence, by \cite[Lemma~2.4.2]{assmus1994designs}, $\delta^\perp\geq (r+\lambda)/\lambda$, where $r=(m-1)\lambda/(\delta-1)$. Thus, $\delta^\perp\geq (m-1)/(\delta-1)+1$, that is, $m-1\leq (\delta-1)(\delta^\perp-1)$ as required. Suppose $C\subseteq C^\perp$. Then $\delta^\perp\leq \delta$ so that $m-1\leq (\delta-1)^2$, that is, $\delta\geq \sqrt{m-1}+1$.
\end{proof}

Mortimer \cite{mortimer1980modular} investigated the permutation modules of $2$-transitive groups, in particular studying when proper submodules exist, other than those corresponding to the binary repetition code and its dual. In the case of permutation modules over $\F_2$, a more thorough account is given in \cite{Ivanov1993}, completing most of the relevant cases left unanswered from \cite{mortimer1980modular}. The results from \cite{Ivanov1993,mortimer1980modular} are crucial for the proof of our next result.

\begin{theorem}\label{binarygroups}
 Let ${\mathcal C}$ be a binary linear-$(X,2)$-neighbour-transitive code in $H(m,2)$ with minimum distance  at least $5$, and ${\mathcal C}\neq\Rep(m,2)$. Then ${\mathcal C}$ contains a binary linear-$(X',2)$-neighbour-transitive code $C$ of dimension $k$ with minimum distance $\delta\geq 5$, where $X'=T_{C}\rtimes X_{\b 0}$, such that the values of $X_{\b 0}$, $m$, $\delta$ and $k$ satisfy one of the lines of Table~\ref{binarytable}.
 
 Also, for $X_{\b 0}$, $m$ and $k$ as in each of the lines of Table~\ref{binarytable}, there exists a binary code $C$ in $H(m,2)$ of $\F_2$-dimension $k$, such that $C$ is linear-$(T_C\rtimes X_{\b 0},2)$-neighbour-transitive, for some $\delta$ satisfying the condition in that line of the table.
\end{theorem}

Note that perfect linear codes over finite fields are classified (see \cite{tietavainen1973nonexistence}), and a code with minimum distance $100$ and length $276$ invariant under $X_{\b 0}^M\cong \Co_3$ is given in \cite{Haemers1993}. Also, the minimum distance $\delta$ of the codes corresponding to line 9 of Table~\ref{binarytable}, where $\soc(X_{\b 0})=\PSU_3(r)$ with $r\equiv 1\pmod{4}$, has not been shown here to satisfy $\delta\geq 5$, though these codes are $2$-neighbour-transitive with $\delta\geq 4$, by Lemma~\ref{moduleiscode}; see Remark~\ref{psuremark}.

\begin{proof}
 By Lemma~\ref{binarycodepermsubmodules}, $X_{\b 0}$ is $2$-homogeneous on $M$ and ${\mathcal C}$ is a submodule of the vertex set $V\cong\F_2^m$ of $H(m,2)$, regarded as the $\F_2$-permutation module for the action of $X_{\b 0}$ on $M$. Every $2$-homogeneous permutation module has (at least) two proper submodules, given by the repetition code $Y$ and its dual $Y^\perp$, under the standard inner product. Since $Y^\perp$ has minimum distance $2$, and ${\mathcal C}\neq Y$ by assumption, we require that the \emph{heart} of $V$, defined to be $Y^\perp/(Y\cap Y^\perp)$, is reducible. 
 
 Lemma~\ref{moduleiscode} implies that any submodule of $V$ of dimension at least $1$, other than $V$,$Y$ and $Y^\perp$, gives a code $C$ such that $C$ is either perfect with $\delta=3$, or $C$ is linear-$(X,2)$-neighbour-transitive with $\delta\geq 4$. As perfect linear codes over $\F_2$ have been classified (see \cite{tietavainen1973nonexistence}, for instance), differentiating these cases is possible. In fact, if $\delta=3$, $m\geq 5$ and $C$ is perfect, then $C$ is a Hamming code of length $2^t-1$, where $t\geq 3$. 
 
 First, suppose $X_{\b 0}$ acts $2$-transitively on $M$. The main result of \cite{mortimer1980modular} then implies that $X_{\b 0}$ and $m$ are as in Table~\ref{binarytable}, so that we can restrict the discussion to those groups listed. Further details taken from \cite{mortimer1980modular} will be pointed out as they are used. The remaining information in Table~\ref{binarytable} comes from explicit examples, with the help of \cite{Ivanov1993}, and Lemma~\ref{mindesigncodedist} for some of the bounds on $\delta$. A \emph{preminimal} submodule of $V$ is defined to be a submodule $U$ containing $Y$ such that $U/Y$ is a minimal submodule of $V/Y$. Hence, we have that ${\mathcal C}$ contains a module $U$ that is either minimal or preminimal. If $\delta_U$ is the minimum distance of $U$, then $\delta\leq \delta_U$ and $\delta\geq 5$ implies $\delta_U\geq 5$. In the following, let $C=U$. In \cite[Section~3]{Ivanov1993} the faithful minimal and preminimal $X_{\b 0}$-submodules of the permutation module $V$ for the $2$-transitive group $X_{\b 0}$ are classified. 
 
 Let $X_{\b 0}$ be a $2$-transitive subgroup of $\AGL_t(2)$ and $m=2^t$, where $t\geq 3$, since $m\geq\delta\geq 5$. By \cite[Theorem~4.1]{Ivanov1993}, there is no minimal submodule and there is a unique preminimal submodule. It is spanned by the constant and linear functions, giving $\delta\leq 2^{t-1}$. Indeed this preminimal submodule is the Reed-Muller code $\RM_2(1,t)$, by \cite[Theorem~5.3.3]{assmus1994designs}, with minimum distance equal to $2^{t-1}$ and dimension $t+1$. Moreover, $\delta\geq 5$ is satisfied when $t\geq 4$, while if $t\leq 3$ then $\delta=2^{t-1}<5$, as in line 2 of Table~\ref{binarytable}. 
 
 Let $\soc(X_{\b 0})\cong \PSL_t(2^k)$ and $m=(2^{kt}-1)/(2^k-1)$, or $X_{\b 0}\cong \alt_7\leq\PSL_4(2)$ and $m=15$. Note that $t\geq 3$, by \cite{mortimer1980modular}. By \cite[Theorems~5.1 and~5.2]{Ivanov1993}, both a unique preminimal submodule, of dimension $m^k+1$, and a unique minimal submodule, of dimension $m^k$, exist, and are generated by the characteristic functions of all hyperplanes, and the characteristic functions of all complements of hyperplanes, respectively. The subfield subcode of the projective Reed-Muller code $\PRM_{2/2^k}(2^k-1,t)$ has minimum distance $(2^{k(t-1)}-1)/(2^k-1)$, by \cite[Proposition~5.7.1]{assmus1994designs}, and is generated by the characteristic functions of all hyperplanes, by \cite[Theorem~5.7.9]{assmus1994designs}. The code generated by the characteristic vectors of all complements of hyperplanes is the even weight subcode of $\PRM_{2/2^k}(2^k-1,t)$ which has minimum distance at least $(2^{k(t-1)}-1)/(2^k-1)+1$, by \cite[Theorem~5.7.9]{assmus1994designs}. If $(k,t)=(1,3)$ then the characteristic vector of the complement of a hyperplane has weight $4$. Thus, $\delta\geq 5$ requires $(k,t)\neq (1,3)$. So line 3 or 4 of Table~\ref{binarytable} holds.
 
 Let $\soc(X_{\b 0})\cong\PSL_2(r)$, where, by \cite{mortimer1980modular}, $X_{\b 0}$ is not $3$-transitive, $r\equiv \pm 1\pmod 8$ and $m=r+1$. By \cite[Lemma~5.4]{Ivanov1993}, there are no minimal submodules and exactly two preminimal submodules, both having dimension $(r+1)/2$ and producing codes with minimum distance at most $(r+1)/2$. These codes are the extended quadratic residue codes by \cite[Theorem~2.10.2 and Corollary~2.10.1]{assmus1994designs}. Perfect codes must have odd length, by \cite{tietavainen1973nonexistence}, which implies that $\delta\geq 4$, but to satisfy $\delta\geq 5$ requires $r\geq 9$. By \cite[Theorem~2.10.1 and Corollary~2.10.2]{assmus1994designs}, the minimum distance of these extended quadratic residue codes satisfies $(\delta-1)^2\geq r$, so that, except for $r=9$ (since $15$ is not a prime power), we have that $\delta\geq 5$ holds. Suppose $r=9$ and $\delta\geq 5$. Then, by \cite[Table~1]{Best78boundsfor}, $|C|\leq 12$. However, by \cite[(F) Page~13]{mortimer1980modular}, $C$ has dimension at least $4$, and thus $|C|\geq 2^4$, giving a contradiction. Thus $\delta\geq 5$ occurs only when $r\geq 23$, and line 5 of Table~\ref{binarytable} holds. Note that when $r=23$ we have that $C$ is the extended binary Golay code (see \cite[Exercise~2.10.1 (4)]{assmus1994designs}).
 
 Let $X_{\b 0}\cong\Sp_{2t}(2)$, $t\geq 2$ and $m=2^{2t-1}-2^{t-1}$ or $2^{2t-1}+2^{t-1}$. By \cite[Theorem~6.2]{Ivanov1993}, there are no minimal submodules; if $t=2$ there are two preminimal submodules, each having dimension $2t+1$, and if $t\geq 3$ then there is a unique preminimal submodule, of dimension $2t+1$. From \cite[Lemma~6.1]{Ivanov1993} we have that $\delta$ is $2^{2t-2}$ and $2^{2t-2}-2^{t-1}$ when $m=2^{2t-1}-2^{t-1}$ and $2^{2t-1}+2^{t-1}$, respectively. Thus, $\delta\geq 5$ requires $t\geq 3$, and line 6 or 7 of Table~\ref{binarytable} holds.
 
 Let $\soc(X_{\b 0})\cong \PSU_3(r)$ and $m=r^3+1$. Then, by \cite{mortimer1980modular}, $r$ is odd. Let $D$ be the design submodule of the $2$-$(r^3+1,r+1,1)$ design invariant under $X_{\b 0}$. If $r\equiv 1\pmod{4}$ then, by \cite[Theorem~7.2]{Ivanov1993}, there are no minimal submodules and two preminimal submodules, namely $D$ and $D^\perp$, of dimensions $r^2-r+1$ and $r^3-r^2+r$, respectively. Since $m$ is even, and hence $D$ is not perfect, $\delta\geq 4$ is satisfied for $D$ and $D^\perp$, by Lemma~\ref{moduleiscode}. Let $r\equiv 3\pmod{4}$. Then, by \cite[Theorem~7.3]{Ivanov1993} and \cite[Theorem~4.1]{HISS2004223}, $D^\perp$ is the unique preminimal submodule of dimension $r^2-r+1$. By Lemma~\ref{mindesigncodedist}, and since $D$ has minimum distance at most $r+1$, it follows that $D^\perp$ has minimum distance at least $r^2+1$, as in lines 8 or 9 of Table~\ref{binarytable}.
  
 Let $\soc(X_{\b 0})\cong\Ree(r)$ and $m=r^3+1$. By \cite[Theorem~7.4]{Ivanov1993}, a unique preminimal submodule of dimension $r^2-r+1$ exists, so that $C\subseteq C^\perp$. Thus, by Lemma~\ref{mindesigncodedist}, $\delta\geq r^{3/2}+1$, and line 10 of Table~\ref{binarytable} holds.
 
 Let $X_{\b 0}\cong \mg_{m}$, where $m=22,23$ or $24$. For $m=24$, the Golay code $\G_{24}$ has minimum distance $8$ and dimension $12$. For $m=23$ the Golay code $\G_{23}$ has minimum distance $7$ and dimension $12$, whilst the even weight subcode of $\G_{23}$ has minimum distance $8$ and dimension $11$. Since $\G_{23}$ has covering radius $3$, puncturing $\G_{23}$ results in a code of length $22$ that is invariant under $\mg_{22}$ with minimum distance $6$. The dual of this code has dimension $10$ and minimum distance $8$, confirmed by calculation in \cite{GAP4}. By the discussion in \cite[Section~8]{Ivanov1993} these are all the possibilities for $C$, so lines 11--13 of Table~\ref{binarytable} hold.
 
 Let $X_{\b 0}=\HS$. Then \cite[Section~8]{Ivanov1993} states that there exists a unique preminimal submodule, which is a codimension $1$ submodule of a module $D$ generated by a $2$-$(176,50,14)$ design. By \cite{CALDERBANK1982233}, $D$ has minimum distance $50$ and dimension $22$. Hence $C$ has minimum distance at least $50$ and dimension $21$ and Table~\ref{binarytable}, line 14 holds. Let $X_{\b 0}\cong \Co_3$. Then \cite[Section~8]{Ivanov1993} states that there exists a unique preminimal submodule of dimension $23$. A code $C$ with length $276$, dimension $23$ and $\delta=100$ is constructed in \cite{Haemers1993}, so Table~\ref{binarytable}, line 15 holds.
  
 This completes the examination of all possibilities for $X_{\b 0}$ $2$-transitive on $M$. Suppose finally that $X_{\b 0}$ is a $2$-homogeneous, but not $2$-transitive, subgroup of $\AGaL_1(r)$ where $m=r$ is a prime power, so that $r\equiv 3 \pmod{4}$, by \cite{Kantor1969}. Let $r\equiv 7\pmod{8}$. Then, by \cite[Lemma~2.10.1 and Theorem~2.10.2]{assmus1994designs}, quadratic residue codes provide examples of dimension $(r-1)/2$. By \cite[Corollary~2.10.1]{assmus1994designs}, these quadratic residue codes satisfy $C=C^\perp$, so that, by Lemma~\ref{mindesigncodedist}, we have $\delta\geq\sqrt{r-1}+1$. Thus $\delta\geq 5$ for $r\geq 23$. Since $15$ is not a prime power, only the possibility that $r=7$ remains. If $r=7$, then the quadratic residue codes are the perfect Hamming code and its dual with minimum distances $3$ and $4$, and dimensions $4$ and $3$, respectively. Note also that when $r=23$ then $C$ is the binary Golay code, by \cite[Exercise~2.10.1 (4)]{assmus1994designs}). Comparing dimensions tells us that these are all the possibilities for $C$ here. The perfect Hamming code does not arise for larger $r$, since, for $t\geq 4$, $\PSL_t(2)$ does not have a subgroup that acts $2$-homogeneously, but not $2$-transitively, on $2^t-1$ points. For $r\equiv 3\pmod{8}$, the argument in \cite{mortimer1980modular} for $\PSL_2(r)\leq G\leq \PSiL_2(r)$ gives a contradiction. Thus, line 1 of Table~\ref{binarytable} holds, completing the proof.
\end{proof}

Theorem~\ref{binaryx2ntchar} may now be proved.

\begin{proof}[Proof of Theorem~\ref{binaryx2ntchar}]
 Let $Y=\Aut({\mathcal C})$. Suppose ${\mathcal C}$ is a $(Y,2)$-neighbour-transitive code in $H(m,2)$ with minimum distance at least $5$. It follows from Proposition~\ref{x12trans} that $Y_i^{Q_i}\cong \s_2$ and from Proposition~\ref{ihom} that $Y$ acts transitively on $M$. Thus, either $Y\cap B$ is trivial, or ${\mathcal C}$ is $Y$-alphabet-affine. If ${\mathcal C}$ is $Y$-alphabet-affine then, by Proposition~\ref{orbitof0is2ntmodule}, ${\mathcal C}$ is a $(Y,2)$-neighbour-transitive extension of an $\F_2 Y_{\b 0}$-submodule $W$ of the vertex set of $H(m,2)$, viewed as the permutation module $\F_2^m$ for $Y_{\b 0}$. If $Y\cap B$ is trivial then ${\mathcal C}$ is as in \cite[Theorem~1.1]{ef2nt}, and if $W=\Rep(m,2)$ then ${\mathcal C}$ is as in Theorem~\ref{onedimensionaltheorem}, which combined give the first and second cases of the result. If $W$ is not the repetition code then the minimum distance of $W$ is less than $m$, by \cite[Lemma~2.15]{ef2nt}. Thus, by Theorem~\ref{binarygroups}, $W$, and hence ${\mathcal C}$, contains a code $C$ with parameters as in Table~\ref{binarytable}, with $X_{\b 0}=Y_{\b 0}$. Since $Y$ acts transitively on $\C$ and, by Part 1 of Proposition~\ref{orbitof0is2ntmodule}, $T_C\leq Y$, the third part of the result holds. 
 
 Conversely, if ${\mathcal C}$ is a code in $H(m,2)$ as in the first or second part of the result, then ${\mathcal C}$ is $2$-neighbour-transitive by \cite[Theorem~1.1]{ef2nt} or Theorem~\ref{onedimensionaltheorem}. Suppose that ${\mathcal C}$ is a code in $H(m,2)$ containing a linear subcode $C$ satisfying the conditions of Part 3 of Theorem~\ref{binaryx2ntchar}. By Theorem~\ref{binarygroups}, $C$ is $2$-neighbour-transitive by, so that, by Proposition~\ref{ihom} and since $\delta\geq 5$, we have that $X_{\b 0}$ acts transitively on $\varGamma_i({\b 0})$, for $i=1,2$. Since $\C$ is a union of a set $\S$ of cosets of $C$, $\Aut(\C)$ acts transitively on $\S$, and $T_C\leq\Aut(\C)$ it follows that $\Aut(\C)$ acts transitively on $\C$. Hence $\C$ is $2$-neighbour-transitive.
\end{proof}


\begin{thebibliography}{10}

\bibitem{assmus1994designs}
E.~F. Assmus and J.~D. Key.
\newblock {\em Designs and their Codes}, volume 103 of {\em Cambridge Tracts in
  Mathematics}.
\newblock Cambridge University Press, 1994.

\bibitem{Best78boundsfor}
M.~R. Best, A.~E. Brouwer, F.~J. Macwilliams, A.~M. Odlyzko, and N.~J.~A.
  Sloane.
\newblock Bounds for binary codes of length less than 25.
\newblock {\em IEEE Trans.~ Information Theory}, pages 81--93, 1978.

\bibitem{borges2000nonexistence}
J.~Borges and J.~Rif{\`a}.
\newblock On the nonexistence of completely transitive codes.
\newblock {\em Information Theory, IEEE Transactions on}, 46(1):279--280, 2000.

\bibitem{borgesrho1}
J.~Borges, J.~Rif{\`a}, and V.~Zinoviev.
\newblock On linear completely regular codes with covering radius $\rho=1$,
  construction and classification.
\newblock {\em arXiv preprint}, (arXiv:0906.0550), 2009.

\bibitem{borges2012new}
J.~Borges, J.~Rif{\`a}, and V.~Zinoviev.
\newblock New families of completely regular codes and their corresponding
  distance regular coset graphs.
\newblock {\em Designs, Codes and Cryptography}, pages 1--10, 2012.

\bibitem{Borges201468}
J.~Borges, J.~Rif{\`a}, and V.~Zinoviev.
\newblock Families of completely transitive codes and distance transitive
  graphs.
\newblock {\em Discrete Mathematics}, 324:68--71, 2014.

\bibitem{borges2017survey}
J.~{Borges}, J.~{Rif{\`a}}, and V.~A. {Zinoviev}.
\newblock On completely regular codes.
\newblock {\em arXiv preprint}, (arXiv:1703.08684), 2017.

\bibitem{brouwer}
A.~E. Brouwer, A.~M. Cohen, and A.~Neumaier.
\newblock {\em Distance-Regular Graphs}, volume~18 of {\em Ergebnisse der
  Mathematik und ihrer Grenzgebiete (3) [Results in Mathematics and Related
  Areas (3)]}.
\newblock Springer-Verlag, Berlin, 1989.

\bibitem{burnside1911theory}
W.~Burnside.
\newblock {\em Theory of groups of finite order}.
\newblock University, 1911.

\bibitem{CALDERBANK1982233}
A.~R. Calderbank and D.~B. Wales.
\newblock A global code invariant under the higman-sims group.
\newblock {\em Journal of Algebra}, 75(1):233 -- 260, 1982.

\bibitem{cameron1991designs}
P.~J. Cameron and J.~H. van Lint.
\newblock {\em Designs, Graphs, Codes and Their Links}.
\newblock London Mathematical Society Student Texts. Cambridge University
  Press, 1991.

\bibitem{delsarte1973algebraic}
P.~Delsarte.
\newblock {\em An Algebraic Approach to the Association Schemes of Coding
  Theory}.
\newblock Philips research reports: Supplements. N.~V.~Philips'
  Gloeilampenfabrieken, 1973.

\bibitem{dixon1996permutation}
J.~D. Dixon and B.~Mortimer.
\newblock {\em Permutation groups}, volume 163.
\newblock New York: Springer, 1996.

\bibitem{GAP4}
The GAP~Group.
\newblock {\em {GAP -- Groups, Algorithms, and Programming, Version 4.9.1}},
  2018.

\bibitem{Gill2017}
N.~Gill, N.~I. Gillespie, and J.~Semeraro.
\newblock Conway groupoids and completely transitive codes.
\newblock {\em Combinatorica}, pages 1--44, 2017.

\bibitem{ef2nt}
N.~I. Gillespie, M.~Giudici, D.~R. Hawtin, and C.~E. Praeger.
\newblock Entry-faithful 2-neighbour transitive codes.
\newblock {\em Designs, Codes and Cryptography}, pages 1--16, 2015.

\bibitem{aas2nt}
N.~I. Gillespie and D.~R. Hawtin.
\newblock Alphabet-almost-simple $2$-neighbour-transitive codes.
\newblock {\em Ars Mathematica Contemporanea}, 14(2):345--357, 2017.

\bibitem{ondimblock}
N.~I. Gillespie, D.~R. Hawtin, and C.~E. Praeger.
\newblock $2$-{N}eighbour-transitive codes with small blocks of imprimitivity.
\newblock {\em arXiv preprints}, (arXiv:1806.10514), June 2018.

\bibitem{ntrcodes}
N.~I. Gillespie and C.~E. Praeger.
\newblock Neighbour transitivity on codes in {H}amming graphs.
\newblock {\em Designs, Codes and Cryptography}, 67(3):385--393, 2013.

\bibitem{gillespieCharNT}
N.~I. Gillespie and C.~E. Praeger.
\newblock Characterisation of a family of neighbour transitive codes.
\newblock {\em arXiv preprint}, (arXiv:1405.5427), 2014.

\bibitem{gillespiediadntc}
N.~I. Gillespie and C.~E. Praeger.
\newblock Diagonally neighbour transitive codes and frequency permutation
  arrays.
\newblock {\em Journal of Algebraic Combinatorics}, 39(3):733--747, 2014.

\bibitem{gillespie2012nord}
N.~I. Gillespie and C.~E. Praeger.
\newblock New characterisations of the {N}ordstrom-{R}obinson codes.
\newblock {\em Bulletin of the London Mathematical Society}, 49(2):320--330,
  2017.

\bibitem{Giudici1999647}
M.~Giudici and C.~E. Praeger.
\newblock Completely transitive codes in {H}amming graphs.
\newblock {\em European Journal of Combinatorics}, 20(7):647 -- 662, 1999.

\bibitem{Haemers1993}
W.~H. Haemers, C.~Parker, V.~Pless, and V.~D. Tonchev.
\newblock A design and a code invariant under the simple group ${C}o_3$.
\newblock {\em J.~ Comb.~ Theory, Ser.~ A}, 62:225--233, 1993.

\bibitem{myphdthesis}
D.~R. Hawtin.
\newblock {\em Algebraic symmetry of codes in hamming graphs}.
\newblock PhD thesis, The University of Western Australia, 2017.

\bibitem{HISS2004223}
G.~Hiss.
\newblock {H}ermitian function fields, classical unitals, and representations
  of $3$-dimensional unitary groups.
\newblock {\em Indagationes Mathematicae}, 15(2):223 -- 243, 2004.

\bibitem{Ivanov1993}
A.~A. Ivanov and C.~E. Praeger.
\newblock On finite affine 2-arc transitive graphs.
\newblock {\em Eur.~ J.~ Comb.}, 14(5):421--444, September 1993.

\bibitem{james2001representations}
G.~James and M.~W. Liebeck.
\newblock {\em Representations and Characters of Groups}.
\newblock Cambridge mathematical textbooks. Cambridge University Press, 2001.

\bibitem{Kantor1969}
W.~M. Kantor.
\newblock Automorphism groups of designs.
\newblock {\em Mathematische Zeitschrift}, 109(3):246--252, 1969.

\bibitem{Kudekar:2016:RCA:2897518.2897584}
S.~Kudekar, S.~Kumar, M.~Mondelli, H.~D. Pfister, E.~\c{S}a\c{s}o\u{g}lu, and
  R.~Urbanke.
\newblock {R}eed-{M}uller codes achieve capacity on erasure channels.
\newblock In {\em Proceedings of the Forty-eighth Annual ACM Symposium on
  Theory of Computing}, STOC '16, pages 658--669. ACM, 2016.

\bibitem{macwilliams1978theory}
F.~J. MacWilliams and N.~J.~A. Sloane.
\newblock {\em The Theory of Error Correcting Codes}.
\newblock North-Holland Mathematical Library. North-Holland, 1978.

\bibitem{mortimer1980modular}
B.~Mortimer.
\newblock The modular permutation representations of the known doubly
  transitive groups.
\newblock {\em Proceedings of the London Mathematical Society}, 3(1):1--20,
  1980.

\bibitem{sole1987completely}
P.~Sol{\'e}.
\newblock Completely regular codes and completely transitive codes.
\newblock RR-0727, 1987.

\bibitem{stinson2004combinatorial}
D.~R. Stinson.
\newblock {\em Combinatorial Designs: Construction and Analysis}.
\newblock Springer, 2004.

\bibitem{tietavainen1973nonexistence}
A.~Tiet{\"a}v{\"a}inen.
\newblock On the nonexistence of perfect codes over finite fields.
\newblock {\em SIAM Journal on Applied Mathematics}, 24(1):88--96, 1973.

\bibitem{Zinoviev73thenonexistence}
A.~Zinoviev and V.~K. Leontiev.
\newblock The nonexistence of perfect codes over {G}alois fields.
\newblock In {\em Problems of Control and Information 2}, pages 123--132, 1973.

\end{thebibliography}
\end{document}